\documentclass[11pt,roman]{amsart}
\usepackage{amsfonts}
\usepackage{amssymb, amsmath, latexsym}
\usepackage{lipsum}
\usepackage{mathtools}

\usepackage{array}
\usepackage{cite}

\setbox0=\hbox{$+$}
\newdimen\plusheight
\plusheight=\ht0
\def\+{\;\lower\plusheight\hbox{$+$}\;}

\setbox0=\hbox{$-$}
\newdimen\minusheight
\minusheight=\ht0
\def\-{\;\lower\minusheight\hbox{$-$}\;}

\setbox0=\hbox{$\cdots$}
\newdimen\cdotsheight
\cdotsheight=\plusheight
\def\cds{\lower\cdotsheight\hbox{$\cdots$}}

\usepackage{mleftright}
\numberwithin{equation}{section}
\theoremstyle{plain}
\newtheorem{theorem}{Theorem}[section]

\newtheorem{lemma}[theorem]{Lemma}

\newtheorem{corollary}[theorem]{Corollary}
\newtheorem{proposition}[theorem]{Proposition}
\newtheorem{remark}[theorem]{Remark}

\newtheorem{definition}[theorem]{Definition}
\newtheorem{example}[theorem]{Example}

\makeatletter
\def\@bignumber#1#2{%
	\ifx#2\end
	#1\let\next\@gobble
	\else
	#1\hspace{0pt plus 1pt}\let\next\@bignumber
	\fi
	\next#2}
\newcommand{\bignumber}[1]{\@bignumber#1\end}
\makeatother

\begin{document}
	\allowdisplaybreaks
	\title[On $\mathcal{I}^\mathcal{K}$-convergence in topological spaces via semi-open sets] {On $\mathcal{I}^\mathcal{K}$-convergence in a topological space\\ via semi-open sets}
	
	\author{Ankur Sharmah}
	
	\address{Department of Mathematical Sciences, Tezpur University, Napam 784028, Assam, India}
	\email{ankurs@tezu.ernet.in}
	
	\author{Debajit Hazarika}
	\address{Department of Mathematical Sciences, Tezpur University, Napam 784028, Assam, India}
	\email{debajit@tezu.ernet.in}
	
	
	
	\maketitle
	\begin{abstract}
		A sequece $\{x_n\}$ is {\tiny $\mathcal{S}$-$\mathcal{I}^\mathcal{K}$}-convergent to $x$, if there exists a set $M \in \mathcal{I}^*$ such that \linebreak $\{n \in M: x_n \notin U  \} \in \mathcal{K}$, for every non empty semi-open set $U$ containing $\xi$. In this \linebreak paper, we introduced the concept of {\tiny $\mathcal{S}$-$\mathcal{I}^\mathcal{K}$}-convergence which is a generalization of {\tiny $\mathcal{S}$-$\mathcal{I}$}-convergence introduced by Guevara et al. \cite{GSR20} and discuss the relation with compact sets. Moreover, we introduce the notion of {\tiny $\mathcal{S}$-$\mathcal{I}^\mathcal{K}$}-cluster point of a sequence and the set of {\tiny $\mathcal{S}$-$\mathcal{I}^\mathcal{K}$}-cluster points of a sequence is characterize as semi-closed subsets. Using this concept we prove that if for a semi-Lindeloff space, each sequence has an $\mathcal{S}$-$\mathcal{I}^\mathcal{K}$-cluster point, then the space is semi-compact.
	\end{abstract}
	
\vspace{1mm}
	
	\section{Introduction and some prelimineries}
	After Kuratowski introduced ideals earlier in 1933, the notion has become familiar as a collection of sets considered to be ``small" or ``negligible". An ideal \cite{HNV04} $\mathcal{I}$ on a set $\mathbb{S}$ is a collection of subsets of $\mathbb{S}$ closed under finite unions and subset inclusion. $Fin$ is a basic ideal defined as the collection of all finite subsets of $\mathbb{S}$. In an ordinary space, three basic topological notions namely, convergence, closure and neighborhood, play a crucial role to ascertain other topological properties. Inevitably, convergence has been studied widespreadly since its inception as well as throughout the last century. In the recent past, ideal theory along with the theory of convergence has been used to bring up some promising generalizations of existing concepts in General Topology. In particular, in 2000, two notions of convergence of a sequence, which are termed as $\mathcal{I}$ and $\mathcal{I}^*$-convergence were introduced by Kostyrko et al. \cite{KSW01} for the set of real numbers  and subsequently, in 2005, for a topological space by Lahiri and Das \cite{LD05}. It is doubtless that subsequently it has been practised actively and out of which some of the works can be found in \cite{TH09, DDGB19, STZ05}. Later, in 2011, Macaz and Sleziak \cite{MS11}  introduced the notion of $\mathcal{I^{\mathcal{K}}}$-convergence of a function in a topological space. Though appeared in the context of $\mathcal{I^{\star}}$-convergence, as a matter of fact, {\tiny $\mathcal{I}^\mathcal{K}$}-convergence further extends the concept of ideal convergence. In particular, if the ideals $\mathcal{I}$ and $\mathcal{K}$ coincide, then so do the notions {\tiny $\mathcal{I}^\mathcal{K}$} and $\mathcal{I}$-convergence. Throughout the last decade, {\tiny $\mathcal{I}^\mathcal{K}$}-convergence has been extensively studied by many researchers \cite{DST14, DSS19, MS11}. On the other hand, in 1963, Levine introduced semi-open sets \cite{L63} in a topological space and afterwards, it has been considered to generalize several concepts in General Topology. Recently, Guevara et al. \cite{GSR20} have used the concept of semi-open set in order to define and study the notion $\mathcal{S}$-$\mathcal{I}$-convergence in topological spaces. In this sequel we define {\tiny $\mathcal{I}^\mathcal{K}$}-convergence using semi-open sets in a topological space and denote it by {\tiny $\mathcal{S}$-$\mathcal{I}^\mathcal{K}$}-convergence. 
	
	\vspace{2mm}
	In this article, we mainly deal with the proper ideals (not containing $\mathbb{N}$) on the set of natural numbers $\mathbb{N}$, that include all the finite subsets of $\mathbb{N}$ (admissible), to study different aspects of {\tiny $\mathcal{S}$-$\mathcal{I}^\mathcal{K}$}-convergence in a topological space.
	
	\vspace{2mm}
	Meanwhile, we refer to \cite{HNV04} for the basic general topological terminologies, definitions and results that are mentioned in the content. Thereafter, unless specified otherwise, we term $X$ to be a topological space and $\mathcal{I}$ and $\mathcal{K}$ to be ideals on $\mathbb{N}$. For an ideal $\mathcal{I} \subset P(\mathbb{N})$, we find two  additional subsets of $P(\mathbb{N})$ namely, $\mathcal{I}^{\star}$ or $F(\mathcal{I})$ and $\mathcal{I}^+$, defined as; $\mathcal{I}^{\star} := \{A \subset \mathbb{N} : A^c \in \mathcal{I}\}$ and $\mathcal{I}^+$:= collection of all subsets which do not belong to $\mathcal{I}$. Two definitions of ideal convergence of a sequence in a topological space, denoted by $\mathcal{I}$ and $\mathcal{I^{\star}}$-convergence, that appeared in \cite{LD05}, are defined as follows. For a topological space X, a sequence $x := \{x_n\}_{n \in \mathbb{N}}$ is said to be $\mathcal{I}$-convergent to $\xi$, denoted by $x_n \rightarrow _\mathcal{I} \xi$, if $\{n \in \mathbb{N} : x_n \notin {U}\} \in \mathcal{I}$, for all neighborhood ${U}$ of $\xi$. In addition, a sequence $x := \{x_n\}_{n \in \mathbb{N}}$ is $\mathcal{I^{\star}}$-convergent to $\xi$ if there exists $M := \{m_1 < m_2 <... < m_k <... \} \in \mathcal{I^{\star}}$ (i.e. $\mathbb{N} \setminus M \in \mathcal{I}$), such that the subsequence $\{x_{m_k}\}$ is convergent to $0$. Lahiri and Das \cite{LD05} proved the equivalence of $\mathcal{I}$ and $\mathcal{I}^*$-convergence, where the space is assumed to be a first axiom $T_1$ space with atleast one limit point. Again, for a given function $f: S \rightarrow X$ on a topological space $X$ which is in fact, a generalization of a sequence, Macaz and Sleziak \cite{MS11} defined $\mathcal{I^{\mathcal{K}}}$-convergence for two ideals $\mathcal{I}$ and $\mathcal{K}$ on $S$.
	
	\begin{definition}{\rm{\cite[Lemma 2.1]{MS11}}}
		{\rm A function $f : S \to X$ is said to be} $\mathcal{I^{\mathcal{K}}}$-convergent { \rm to $x \in X$, if there exists a set $M \in F(\mathcal{I})$ such that the function $g : S \to X$ given by $g(s) = f(s)$, if $s \in M$ and $g(s)= x$, if $s \notin M$, is $\mathcal{K}$-convergent to $x$. If $f$ is $\mathcal{I}^{\mathcal{K}}$-convergent to $x$, then we write $\mathcal{I^{\mathcal{K}}}-\lim f = x$.}
	\end{definition}
	
	\begin{proposition}{\rm{\cite[Proposition 2.1]{SH21}}} \label{Lp}
	Let $X$ be a topological space and $f : S \rightarrow X$ be a function. Let $\mathcal{I}, \mathcal{K}$ be two ideals on $S$ such that $\mathcal{I} \cup \mathcal{K}$ is an ideal. Then
		\begin{itemize}
			\item[\textup{(i)}]	$ f(s) \rightarrow_{\mathcal{I}^{\mathcal{K}^*}} x \equiv f(s) \rightarrow_{(\mathcal{I} \cup \mathcal{K})^*} x$.
			\item[\textup{(ii)}] $f(s) \rightarrow_\mathcal{\mathcal{I}^{\mathcal{K}}} x$ implies $f(s) \rightarrow_{\mathcal{I} \cup \mathcal{K}} x$.
		\end{itemize}
	\end{proposition}
	
	\begin{theorem}{\rm{\cite[Theorem 3.11]{SH21}}} \label{IKc}
		Every continuous function preserves  $\mathcal{I}^\mathcal{K}$-convergence.
	\end{theorem}
	Some of the definitions and results in \cite{GSR20}, \cite{L63}, \cite{CH72}, \cite{M11}, \cite{MP75} that are used in the content of the accompanying sections are listed below. 
	\vspace{2mm}
	\begin{definition} \label{D1}
		{\rm Let $X$ be a topological space. Then}
		\begin{itemize}
			\item[(1)] 	{\rm$O \subset X$ is said to be} semi-open, {\rm if there exists an open set $U$ such that $U \subset A \subset \overline{U}$. The collection of all semi-open subsets of $X$ is denoted by $SO(X)$}.
			\item[(2)] {\rm The complement of a semi-open set is called a} semi-closed set.
			\item[(3)] {\rm The} semi-closure {\rm of a subset $F$ of $X$, denoted by} $sCl(F)$, { \rm is defined as the intersection of all semi-closed set containing $F$. Otherwise, a point $x \in sCl(A)$ if and only if for every semi-open set $U$ containing $x$, $U \cap A \neq \phi$}.
			\item[(4)] { \rm A function $f: X \rightarrow Y$ is said to be} irresolute, { \rm  if $f^{-1}(O) \in SO(X)$ for each $O \in SO(Y)$}. 
			\item[(5)] {\rm A function $f : X \rightarrow Y$ is said to be} semi-continuous, { \rm if $f^{-1}(O) \in SO(X)$ for each open set $O \in Y$}.
			\item[(6)] { \rm A topological space $X$ is said to be} semi-Hausdorff, { \rm if for every distinct pair $x$, $y \in X$, there exist disjoint semi-open set $U$ containing $x$ and $y$ respectively}.
		\end{itemize}
	\end{definition}
	\vspace{1mm}
	\begin{theorem} \label{sc}
		 {\rm{\cite{L63}}}	A function $f: X \rightarrow Y$ is semi-continuous if and only if for each $x \in X$ and each open $V$ in $Y$ containing $f(x)$, there exist $U \in SO(X)$ such that $x \in U$ and $f(U) \subset V$.
	\end{theorem}
	
	\begin{theorem} \label{ir}
		{\rm{\cite{GSR20}}}	A function $f: X \rightarrow Y$ is irresolute if and only if for each $x \in X$ and each $V \in SO(Y)$ containing $f(x)$, there exist $U \in SO(X)$ such that $x \in U$ and $f(U) \subset V$.
	\end{theorem}
	
	\begin{definition}{\rm{\cite[Definition 3.4]{GSR20}}} \label{wac}
		{\rm Let $X$ be a space and $A \subset X$. We say that $p \in X$ is a} semi-$\omega$-accumulation point {\rm of $A$ if for every semi-open set $U$ containing $p$, $U \cap A$ is an infinite set.}
	\end{definition}
	
	\begin{theorem}{\rm{\cite[Theorem 3.5]{GSR20}}} \label{clap}
		Let $X$ be a space and $\mathcal{I}$ be a proper ideal. If every \linebreak sequence $\{x_n\}$ in $X$ has $\mathcal{S}$-$\mathcal{I}$-cluster point, then every infinite subset of $X$ has a \linebreak semi-$\omega$-accumulation point. The
		converse is true if $\mathcal{I}$ does not contain any infinite sets.
	\end{theorem}
	
	\section{Some properties of $\mathcal{S}$-$\mathcal{I}^\mathcal{K}$-convergence}
		For a space $X$, a sequence $x=\{x_n\}$ is said to be $\mathcal{S}$-$\mathcal{I}$-convergent \cite{GSR20} to a point $\xi \in X$ if for every non-empty semi-open set $U$ containing $\xi$, the set $\{n \in \mathbb{N} : x_n \in U \}$ belongs to $\mathcal{I}$. Then, $\xi$ is said to be the $\mathcal{S}$-$\mathcal{I}$-limit of $\{x_n\}$ and is denoted by $\mathcal{S}$-$\mathcal{I}$-$\lim x_n = \xi$. In this section, we introduce the concept of $\mathcal{S}$-$\mathcal{I}^\mathcal{K}$-convergent sequence in a topological space and study some of its properties.
	
	\begin{definition}
		{\rm A sequence $\{x_n\}$ is said to be} $\mathcal{S}$-$\mathcal{I}^\mathcal{K}$-convergent {\rm to a point $ \xi \in X$ if there exists a set $M \in \mathcal{I}^*$ such that for every non empty semi-open set $U$ containing $\xi$, we have $\{n \in M: x_n \notin U  \} \in \mathcal{K}$.} {\rm Subsequently, $\xi$  is said to be a} $\mathcal{S}$-$\mathcal{I}^\mathcal{K}$-limit {\rm of $\{x_n\}$ and is denoted by $\mathcal{S}$-$\mathcal{I}^\mathcal{K}$-$\lim x_n = \xi $.}
	\end{definition}
	
	As usual, if the ideal $\mathcal{K}$ doesnot contain any infinite set, then above definition speaks about $\mathcal{S}$-$\mathcal{I}^*$-convergence. Again, $\mathcal{K}$-convergence implies $\mathcal{I}^\mathcal{K}$-convergence \cite{MS11}, analogously, we have the following lemma.
	\begin{lemma} 
		 $\mathcal{S}$-$\mathcal{K}$ convergence implies $\mathcal{S}$-$\mathcal{I}^\mathcal{K}$-convergence.
		\end{lemma}
	For the sequential criteria, we extend the Proposition \ref{Lp} for semi-open sets.
	\begin{proposition}
		Let $\{x_n\}$ be a sequence in $X$. For $\mathcal{I}, \mathcal{K}$ be two ideals on $\mathbb{N}$ such that $\mathcal{I} \cup \mathcal{K}$ is an ideal. Then 
		
		\begin{itemize}
			\item[\textup{i)}]	$\{x_n\}$ is $\mathcal{S}$-$\mathcal{I}^{\mathcal{K}^*}$ converges $\xi$ if and only if $\{x_n\}$ is $\mathcal{S}$-$(\mathcal{I} \cup \mathcal{K})^*$-converges to $\xi$.
			
			\item[\textup{ii)}] Also, $\{x_n\}$ is $\mathcal{S}$-$\mathcal{I}^{\mathcal{K}}$-converges to $\xi$ implies $\{x_n\}$ is $\mathcal{S}$-$\mathcal{I} \cup \mathcal{K}$-converges to $\xi$.
			
		\end{itemize}
	\end{proposition}
	
	\begin{proof}
		Readers may refer to the proof of Theorem \ref{Lp} for sequences, considering the semi-opensets over open sets in the arguments.
	\end{proof}
	
	\begin{remark} \label{last}
		{ \rm In a space $X$, we have the observations stated below.}
		\begin{itemize}
			\item[1.] {\rm If $\mathcal{I} = \mathcal{K}$, then $\mathcal{S}$-$\mathcal{I}^\mathcal{K}$-convergence conincide with that of $\mathcal{S}$-$\mathcal{I}$-convergence.}
			\item[2.] {\rm If $\mathcal{K}= Fin$, then $\mathcal{S}$-$\mathcal{I}^\mathcal{K}$-convergence implies $\mathcal{S}$-$\mathcal{I}$-convergence.} 
			\item[3.] {\rm If $\mathcal{I} =\mathcal{K} = Fin$, then $\mathcal{S}$-$\mathcal{I}^\mathcal{K}$-convergence coincides with $\mathcal{S}$-convergence and hence, further implies usual convergence.}
		\end{itemize}
	\end{remark}
	The following example shows that even if $\mathcal{I} =\mathcal{K} = Fin$, usual convergence does not implies \linebreak $\mathcal{S}$-$\mathcal{I}^\mathcal{K}$-convergence.
	\begin{example}
		Let $\mathcal{I}$ and $\mathcal{K}$ be two ideals in $\mathbb{N}$ such that $\mathcal{I} \cup \mathcal{K}$ is an ideal. Let $\mathbb{R}$ be the set of real numbers with usual topology and $\{x_n \}$ be defined as $x_n = (-1)^n (\frac{1}{n})$. Then, $x_n \rightarrow 0$. Consider the semi-open set $U = (-1, 0]$ containing $0$. But $\{n \in \mathbb{N}: x_n \notin U \} = \mathbb{N} \notin \mathcal{I} \cup \mathcal{K}$ for any ideals $\mathcal{I}$ and $\mathcal{K}$. So, $\{x_n \}$ is not $\mathcal{S}$-$\mathcal{I} \cup \mathcal{K}$-convergent to $0$. Thus, $\{x_n \}$ is not $\mathcal{S}$-$\mathcal{I}^\mathcal{K}$-convergent to $0$.
	\end{example}
	
	\begin{lemma} \label{L1}
		 $\mathcal{S}$-$\mathcal{I}^\mathcal{K}$-convergence implies $\mathcal{I}^\mathcal{K}$-convergence, for any proper ideals $\mathcal{I}$ and $\mathcal{K}$.
	\end{lemma}
	
	\begin{proof}
		This proof follows immediately from the definition of $\mathcal{S}$-$\mathcal{I}^\mathcal{K}$-convergence.
	\end{proof}
	
	\vspace{1mm}
	The following example shows that the converse of Lemma \ref{L1} is not necessarily true.
	\begin{example}
	 Let $\mathcal{I}$ and $\mathcal{K}$ be two ideals in $\mathbb{N}$ such that $\mathcal{I} \cup \mathcal{K}$ is an ideal. Let $[-1, 1]$ be the interval in $\mathbb{R}$, with usual subspace topology and $\{x_n\}$ be a sequence defined as $x_n = (\frac{1}{n}) \sin (\frac{1}{n})$. So, for any open set $U$ containing $0$, we have $\{n \in \mathbb{N}: x_n \notin U \}$ is finite, that implies  $x_n \rightarrow_\mathcal{K} 0$. Hence, $x_n$ is $\mathcal{I}^\mathcal{K}$-convergent to $0$. Now, consider the semi-open set $V= (-1, 0]$ in $[-1, 1]$. Then $\{n \in \mathbb{N}: x_n \notin V \} = \mathbb{N} \notin \mathcal{I} \cup \mathcal{K}$. Hence, $x_n \nrightarrow_{\mathcal{I} \cup \mathcal{K}} 0$. Thus, Theorem \ref{Lp} (ii) suggests that $\{x_n \}$ is not $\mathcal{I}^\mathcal{K}$-convergent to $0$.
	\end{example}
		
	\begin{theorem}
		In a semi-Hausdorff space $X$, each $\mathcal{S}$-$\mathcal{I}^\mathcal{K}$-convergent sequence has unique $\mathcal{S}$-$\mathcal{I}^\mathcal{K}$-limit in $X$, provided $\mathcal{I} \cup \mathcal{K}$ is an ideal.
	\end{theorem}

\begin{proof}
	Consider a $\mathcal{S}$-$\mathcal{I}^\mathcal{K}$-convergent sequence $\{x_n\}$ in a semi-Hausdorff space $X$. Suppose that $\{x_n\}$ has two distinct $\mathcal{S}$-$\mathcal{I}^\mathcal{K}$-limits, say $a$ and $b$. Being a semi-Hausdorff space, there exist $U,V \in SO(X)$ with $U \cap V = \phi$ such that $a \in U$, $b \in V$. As  $\{x_n\}$ is $\mathcal{S}$-$\mathcal{I}^\mathcal{K}$-convergent to $a$ and $b$, there exist $M_1, M_2 \in \mathcal{I}^*$ such that $\{n \in M_1 : x_n \notin U \} \in \mathcal{K}$ and $\{n \in M_2 : x_n \notin V \} \in \mathcal{K}$. So, for $M = M_1 \cap M_2 \in \mathcal{I}^*~(\neq \phi)$, the sets $\{n \in M : x_n \notin U \}$ and $\{n \in M : x_n \notin V \}$ belong to $\mathcal{K}$. Now, we have
	\begin{center}
		 $\{n \in M : x_n \notin U \cap V \} = \{n \in M : x_n \notin U \} \cup \{n \in M: x_n \notin V \} \in \mathcal{K}$.
	\end{center}
 Again, 
	\begin{center}
		$\{n \in \mathbb{N} : x_n \notin U \cap V \} \subseteq \{n \in M : x_n \notin U \cap V  \} \cup M^\complement \in \mathcal{I} \cup \mathcal{K}$.
	\end{center} 
 But, $\mathcal{I} \cup \mathcal{K}$ is an ideal, which implies $\{n \in \mathbb{N} : x_n \notin U \cap V \} \neq \mathbb{N}$. Therefore, we conclude that $\{n \in \mathbb{N} : x_n \in U \cap V \} \neq \phi$, which is a contradiction.
\end{proof}
	\begin{corollary}
		In a Hausdorff space $X$, each  $\mathcal{S}$-$\mathcal{I}^\mathcal{K}$-convergent sequence has a unique $\mathcal{S}$-$\mathcal{I}^\mathcal{K}$-limit in $X$, provided $\mathcal{I} \cup \mathcal{K}$ is an ideal.
	\end{corollary}
	
	\begin{theorem}
		If $~\mathcal{I}$ and $\mathcal{K}$ be two ideals and $F \subset X$. If there exists an $F$-valued sequence $\{x_n\}$ {\rm (}with distinct elements{\rm )} which is $\mathcal{S}$-$\mathcal{I}^\mathcal{K}$-convergent to $\xi \in X$, then $\xi$ is a semi-limit point of $F$, in essence, $\xi \in sCl(F)$, the semiclosure of $F$.
	\end{theorem}

	\begin{proof}
	Let $U$ be any semi-open subset of $X$ containing the point $\xi$. Since $\{x_n\}$ is $\mathcal{S}$-$\mathcal{I}^\mathcal{K}$-convergent to $\xi \in X$, so, there exists a set $M \in \mathcal{I}^*$ such that $\{n \in M : x_n \notin U \} \in \mathcal{K}$. In other words, $\{n \in M: x_n \in U \} \notin \mathcal{K}$ since $\mathcal{K}$ is a proper ideal. Then choose $n_0 \in \{n \in M: x_n \in U \}$ such that $x_{n_0} \neq \xi$, then $x_{n_0} \in F \cap ( U - \{\xi \})$ and hence, $F \cap (U - \{\xi \}) \neq \phi$. This shows that $\xi$ is a semi-limit point of $F$.
	\end{proof}
	\begin{corollary}\label{sca}
	Let $~\mathcal{I}$ and $\mathcal{K}$ be two ideals and consider, $F \subset X$. If there exists an $F$-valued sequence $\{x_n\}$ {\rm (}with distinct elements{\rm )} which is $\mathcal{S}$-$\mathcal{I}^\mathcal{K}$-convergent to $\xi \in X$, then $\xi \in Cl(F)$.
	\end{corollary}
	
	\begin{theorem}
		If $F \subset X$ is a semi-closed set, then for any sequence in $F$ which is \linebreak $\mathcal{S}$-$\mathcal{I}^\mathcal{K}$-convergent to $a$, we have $a \in F$.
	\end{theorem}
	\begin{proof}
		Suppose that $F \subset X$ is a semi-closed set and $\{x_n\}$ is any sequence in $F$ which is \linebreak $\mathcal{S}$-$\mathcal{I}^\mathcal{K}$-convergent to the point $a$, but $a \notin F$. As $F$ is semi-closed, we have $sCl(F) = F$ and thus, $a \notin sCl(F)$. Then there exists a semi-open set $U$ containing $a$ such that $F \cap U \neq \phi$. As we assume, $\{x_n \}$ is $\mathcal{S}$-$\mathcal{I}^\mathcal{K}$-convergent to $a$, there exist $M \in \mathcal{I}^*$ such that $\{n \in M: x_n \notin U \} \in \mathcal{K}$. Further $\{n \in M: x_n \in U \} \notin \mathcal{K}$, which implies that $\{n \in M : x_n \in U \} \neq \phi$. By our hypothesis, $x_n \in F$, that implies $F \cap U \neq \phi$, which is a contradiction.
	\end{proof} 

	\begin{theorem}\label{scp}
		Let $f: X \rightarrow Y$ be a semi-continuous function. If $\{x_n \}$ is a sequence in $X$ which is $\mathcal{S}$-$\mathcal{I}^\mathcal{K}$-convergent to $\xi \in X$, then $\{f(x_n) \}$ is an $\mathcal{I}^\mathcal{K}$-convergent sequence to $f(\xi)$. 
	\end{theorem}
		
	\begin{proof}
		Consider a sequence $\{x_n \}$ in $X$ which is $\mathcal{S}$-$\mathcal{I}^\mathcal{K}$-convergent to $\xi \in X$.  We claim that $\{f(x_n) \}$ is $\mathcal{S}$-$\mathcal{I}^\mathcal{K}$-convergent to $f(\xi)$. Suppose not, that is $\{f(x_n)\}$ is not $\mathcal{S}$-$\mathcal{I}^\mathcal{K}$-convergent to $f(\xi)$. Then there exists a open set $V \in Y$, containing $f(\xi)$ and for each $M \in \mathcal{I}^*$, we have $\{n \in M : f(x_n) \notin V \} \notin \mathcal{K}$. Now by Theorem \ref{sc}, there exist $U \in SO(X)$ such that $f(U) \subset V$. Now, 
		\begin{center}
			$\{n \in M : f(x_n) \notin V \} \subset \{n \in M : x_n \notin U \}$.
		\end{center}
		Then $\{n \in M : x_n \notin U \} \notin \mathcal{K}$, that implies $\{x_n\}$ is not $\mathcal{S}$-$\mathcal{I}^\mathcal{K}$-convergent to $\xi$. That is a contradiction to our assumption. Hence, $\{f(x_n) \}$ is $\mathcal{S}$-$\mathcal{I}^\mathcal{K}$-convergent to $f(\xi)$.
	\end{proof}

	\begin{theorem}
		Let $f: X \rightarrow Y$ be an irresolute function. If $\{x_n \}$ is a sequence in $X$ which is \linebreak $\mathcal{S}$-$\mathcal{I}^\mathcal{K}$-convergent to $\xi \in X$, then $\{f(x_n) \}$ $\mathcal{S}$-$\mathcal{I}^\mathcal{K}$-converges to $f(\xi)$.
	\end{theorem}

	\begin{proof}
		The proof is similar to that of Theorem \ref{scp} with the use of the characterization of an irresolute function that is shown in Theorem \ref{ir}. 
	\end{proof}
	
\section{$\mathcal{S}$-$\mathcal{I}^\mathcal{K}$-cluster points and compactness}

	 In this section, we introduce the terminology of $\mathcal{I}^\mathcal{K}$-cluster point of a sequence for semi open sets in a topological space. We say a point $p$ in a space $X$ is said to be $\mathcal{I}^*$-cluster point of sequence in $X$ if there exists $M =\{ m_1, m_2,..., m_k,...\}\in \mathcal{I}^*$  such that the subsequence $\{x_{m_k}\}$ has a cluster point $p$, elaborately, for any open set $U$ containing $p$, the set $\{n \in \mathbb{N} : x_{m_k} \in U \}$ is infinite subset of $\mathbb{N}$.
	
	\begin{definition}\label{SIKC}
		{\rm Let $X$ be a space and $\{x_n\}$ be a sequence in $X$. A point $p \in X$ is called a }  $\mathcal{S}$-$\mathcal{I}^\mathcal{K}$-cluster point {\rm of $\{x_n\}$ in $X$ if there exists $M \in \mathcal{I}^*$ such that for any $U \in SO(X)$ containing $p$,  we have $\{n \in M : x_n \in U \} \notin \mathcal{K}$.}
	\end{definition}
	
	\begin{remark}\label{ikcr}
		{\rm Let $X$ be a space and $\{x_n \}$ be a $X$-valued sequence. Then}
		\begin{itemize}
			\item[1.] {\rm If $\mathcal{I} = \mathcal{K}$, then $\mathcal{S}$-$\mathcal{I}^\mathcal{K}$-cluster points of $\{x_n \}$ conincide with that of $\mathcal{S}$-$\mathcal{I}$-cluster points.}
			\item[2.] {\rm If $\mathcal{K}= Fin$, then $\mathcal{S}$-$\mathcal{I}^\mathcal{K}$-cluster points of $\{x_n \}$ implies that of $\mathcal{S}$-$\mathcal{I}$-cluster points.}
		\end{itemize}
	\end{remark}
	
	We denote the collection of all $\mathcal{S}$-$\mathcal{I}^\mathcal{K}$-cluster points of a sequence $x=\{x_n\}$ in a space by $sC_x(\mathcal{I}^\mathcal{K})$. Consequently, from Definition \ref{SIKC}, it straightaway follows that $sC_x(\mathcal{I}^\mathcal{K}) \subseteq sC_x(\mathcal{K})$. 
	
	\begin{lemma}\label{lem}
		$sC_x(\mathcal{I}\cup \mathcal{K}) \subseteq sC_x(\mathcal{I}^\mathcal{K})$, for two ideals $\mathcal{I}$ and $\mathcal{K}$ such that $\mathcal{I} \cup \mathcal{K}$ is an ideal.
	\end{lemma}
	
	\begin{proof}
		Let $y$ be not a $\mathcal{S}$-$\mathcal{I}^\mathcal{K}$-cluster point of $x= \{x_n\}$. Then  for all $M \in \mathcal{I}^*$ such that \linebreak $\{n \in M: x_n \in U \} \in \mathcal{K}$. Hence, $\{n \in \mathbb{N} : x_n \in U \} \subseteq \{n \in M: x_n \in U \} \cup M^\complement \in \mathcal{I} \cup \mathcal{K}$. Hence, $y$ is not a $\mathcal{S}$-$(\mathcal{I} \cup \mathcal{K})$-cluster point of $x$.
	\end{proof}
	
	\vspace{1mm}
	\noindent Recall that a subset $D$ of a topological space $X$ is said to be dense in $X$ if for any open set $U$, $U \cap D \neq \phi$. In an arbitrary space, the notion of dense set for semi-open sets is equivalent to that of open sets. 
	
	\begin{theorem} {\rm{\cite[Theorem 2.4]{M11}}} \label{DSD}
		Let $X$ be a space and $D \subset X$. Then $D$ is dense in $X$ if and only if $U \cap D \neq \phi$ for every $U \in SO(X)$.
	\end{theorem}
	
	\vspace{1mm}
	\noindent Let $X$ be an arbitrary space. Then for a sequence $x=\{x_n\}$ in $X$, we characterize the set $sC_x(\mathcal{I}^\mathcal{K})$ as a semi-closed subsets of $X$. We say, $X$ is a semi-closed hereditarily separable space if every semi-closed subsets of $X$ is separable.
	
	\vspace{1mm}
	\begin{theorem}\label{11}	
		Let $\mathcal{I}$, $\mathcal{K}$ be two ideals on $\mathbb{N}$ and $X$ be a space. Then
		\begin{itemize}
			\item[\textup{(i)}]  For $x=\{x_n\}_{n \in \mathbb{N}}$, a sequence in $X$; $sC_x(\mathcal{I}^\mathcal{K})$ is a semi-closed set.
			\item[\textup{(ii)}] If $(X, \tau)$ is semi-closed hereditarily separable and there exists a disjoint sequence of sets $\{D_n\}$ such that $D_n \subset \mathbb{N}$, $D_n \notin \mathcal{I}, \mathcal{K}$ for all $n$, then for every non empty semi-closed subset $F$ of $X$, there exists a sequence $x$ in $X$ such that $F = sC_x(\mathcal{I}^\mathcal{K})$ provided $\mathcal{I} \cup \mathcal{K}$ is an ideal.
		\end{itemize}
	\end{theorem}
	\begin{proof}
		Consider the sequence $x=\{x_n\}$ in $X$ and $\mathcal{I}$, $\mathcal{K}$ be the two ideals on $\mathbb{N}$.
		\begin{itemize}
			\item[\textup{(i)}] Let $y \in  sCl(C_x(\mathcal{I}^\mathcal{K}))$; the semi-derived set of $C_x(\mathcal{I}^\mathcal{K})$. Let $U$ be an semi-open set containing $y$. It is clear that $U \cap C_x(\mathcal{I}^\mathcal{K}) \neq \phi $. Let $q \in (U \cap C_x(\mathcal{I}^\mathcal{K}))$ i.e., $q \in U$ and $q \in C_x(\mathcal{I}^\mathcal{K})$. Now there exist a set $M \in \mathcal{I}^*$, we have $\{n \in M : y_n \in U\} \notin \mathcal{K}$. Thus, $y \in C_x(\mathcal{I}^\mathcal{K})$.
			\vspace{1mm}
			\item[\textup{(ii)}] $F$ is separable as a semi-closed subset of $X$. Then by Definition [\ref{D1} (3)] and Theorem \ref{DSD}, there exists a countable set $S=\{s_1, s_2,...\} \subset F$ such that $sCl(S)=F$. Consider $x_n = s_i$ for $n \in D_i$. Then, we have a subsequence $\{k_n\}$ of the sequence $\{n\}$. Now, consider the sequence $x=\{x_{n_k}\}$ and let $y\in sC_x(\mathcal{K})$ (taking $y \neq s_i$ otherwise if $y= s_i$ for some $i$, then $y$ is eventually in $F$). We claim that $sC_x(\mathcal{I}^\mathcal{K}) \subset F$. Let $U$ be any semi-open set containing $y$. Then $\{n: x_{n_k} \in U\} \notin \mathcal{K}$ $(\neq \phi)$. So, $s_i \in U$ for some $i$. Therefore, $F \cap U$ is non empty, So $y$ is a semi-limit point of $F$ and semi-closedness of $F$ implies $y \in F$. Hence $sC_x(\mathcal{K}) \subset F$. Further $sC_x(\mathcal{I}^\mathcal{K}) \subseteq  sC_x(\mathcal{K}) \subset F$. For the converse argument, consider $a \in F$ and $U$ be a semi-open set containing $a$, then there exists $s_i \in S$ such that $s_i \in U$. Then $\{n: x_{n_k} \in U\} \supset D_i$ ($\notin \mathcal{K}$, $\mathcal{I}$). Thus $\{n: x_{n_k} \in U\} \notin (\mathcal{I}\cup \mathcal{K})$ i.e., $a \in sC_x(\mathcal{I} \cup \mathcal{K})$. On the otherhand, by lemma \ref{lem}, $sC_f(\mathcal{I} \cup \mathcal{K}) \subseteq sC_f(\mathcal{I}^\mathcal{K})$. Thus, we get the reverse implication. 
		\end{itemize}
	\end{proof}
	
	\begin{remark}
		Theorem \ref{11} extends Theorem 10 in {\rm \cite{LD05}} to semi-open sets, it follows by letting $\mathcal{I}$, $\mathcal{K}$ coincide and using open sets instead of semi-open sets in the above theorem.
	\end{remark}
	\vspace{1mm}
	\begin{theorem}
		Let $X$ be a space and $\mathcal{I}$, $\mathcal{K}$ be two ideals. If every sequence $\{x_n\}$ in $X$ has a $\mathcal{S}$-$\mathcal{I}^\mathcal{K}$-cluster point, then every infinite subset of $X$ has a semi-$\omega$-accumulation point. The converse is also true if $\mathcal{I}$, $\mathcal{K}$ doesnot contain any infinite sets.
	\end{theorem}
	
	\begin{proof}
		Let $F$ (infinite) $\subset X$, then there exists a sequence $\{x _n\}$ of distinct points in $F$. Suppose that every sequence in $X$ has a $\mathcal{S}$-$\mathcal{I}^\mathcal{K}$-cluster point. Let $a$ be a $\mathcal{S}$-$\mathcal{I}^\mathcal{K}$-cluster point of $\{x_n\}$. Then, for $U \in SO(X)$ containing $a$, we have $\{n \in M: x_n \in U \} \notin \mathcal{K}$. Since as per assumption, $Fin \subset \mathcal{K}$ then the set $\{n \in M: x_n \in U \}$ is infinite. Since $x_n \in F$, $U \cap F$ is a infinite set. Hence $a$ is semi-$\omega$-accumulation point of $F$.
		
		Conversely, if $\mathcal{I}$, $\mathcal{K}$ does not contain any infinite subset of  $\mathbb{N}$, then the result follows immeditely by considering Remark \ref{ikcr} and Theorem \ref{clap}.
	\end{proof}
	
	\begin{corollary}
		Let $X$ be a space. If every sequence $\{x_n \}$ has a $\mathcal{S}$-$\mathcal{I}^\mathcal{K}$-cluster point, then every infinite subset of $X$ has a $\omega$-accumulation point. 
	\end{corollary}
	
	\vspace{2mm}
	Considering $X$ to be a topological space, we have the accompanying results showing contrast between the different notions of compactness and $\mathcal{S}$-$\mathcal{I}^\mathcal{K}$-convergence mode.

	\begin{proposition}
		Let $\mathcal{I}$, $\mathcal{K}$ be two ideals on $\mathbb{N}$ that doesnot contain infinite sets. If any sequence $\{x_n\}$ in $X$ has a subsequence which is $\mathcal{S}$-$\mathcal{I}^\mathcal{K}$-convergent, then $X$ is sequentially compact space.
	\end{proposition}

	\begin{proof}
		 This proof follows immediately as a consequence of Remark \ref{last} (iii).
	\end{proof}

	\begin{proposition}
		Let $\mathcal{I}$, $\mathcal{K}$ be two ideals on $\mathbb{N}$. If for any infinite subset $F$ of $X$, there exists a $F$-valued distinct sequence $\{x_n\}$, which is $\mathcal{S}$-$\mathcal{I}^\mathcal{K}$-convergent in $X$, then $X$ is a limit point compact space.
	\end{proposition}

	\begin{proof}
		This proof follows immediately as a consequence of Corollary \ref{sca}.
	\end{proof}
	\vspace{1mm}
	\begin{definition} \label{SC}
	{\rm A space $X$ is said to be:}
		\begin{itemize}
			\item[(1)] semi-compact {\rm{\cite{GSR20}}} {\rm  if every semi-open cover of $X$ possesses a finite subcover.}
			\item[(2)] semi-Lindeloff {\rm{\cite{GSR20}}} {\rm if every semi-open cover of $X$ possesses a countable subcover.}
		\end{itemize}
	\end{definition}

	\begin{theorem}
		If $X$ is a semi-Lindeloff space and each $X$-valued sequence has an $\mathcal{S}$-$\mathcal{I}^\mathcal{K}$-cluster point, then $X$ is a semi-compact space.
	\end{theorem}

	\begin{proof}
		Suppose that $X$ is a semi-Lindeloff space and each $X$-valued sequence has a $\mathcal{S}$-$\mathcal{I}^\mathcal{K}$-cluster point. Let $\mathcal{S}=\{S_\lambda : \lambda \in \Lambda \}$ be a semi-open cover of $X$. So, by Definition \ref{SC}, there exists a countable subcover $\mathcal{S^{'}}= \{S_1, S_2, ..., S_m,...\}$. If possible, consider the sequence $U=\{U_m \}$ such that $U_1 = S_1$ and for each $m >1$, let $U_m =S_m$, where $S_m$ is the first member of the sequence of $U$'s such that $S_m \nsubseteq U_1 \cup U_2 \cup ... \cup U_{m-1}$. Assuming the axiom of choice, consider a sequence $x=\{x_m \}$ such that $x_1 \in U_1$ and for each $m >1$, let $x_m \in U_m - (U_1 \cup U_2 \cup ... \cup U_{m-1})$. Now by our hypothesis, $\{x_m\}$ has a $\mathcal{S}$-$\mathcal{I}^\mathcal{K}$-cluster point, say $l$. Then, there exist i such that $l \in U_j$. Consequently, there exists a set $M \in \mathcal{I}^*$ such that $\{n \in M : x_n \in U_j \} \notin \mathcal{K}$. So, the set $\{n \in M : x_n \in U_j \}$ must be a infinite subset of $\mathbb{N}$. Thus, there exists $k > j$ such that $ k \in \{n \in M : x_n \in U_j \}$; that is $x_k \in U_j$, which is a contradiction.  Subsequently, there must exist $m_0 \in \mathbb{N}$ such that the process of induction for $\{U_m \}$ is impossible to continue after $m=m_0$. Therefore, $\{U_1, U_2, ..., U_{m_0} \}$ is a finite subcover of $X$ for given cover $\mathcal{S}$.
	\end{proof}

	\begin{corollary}
		If $X$ is a semi-Lindeloff space and each $X$-valued sequence has an $\mathcal{S}$-$\mathcal{I}^\mathcal{K}$-cluster point, then $X$ is a compact space.
	\end{corollary}

\section{$\mathcal{I}^\mathcal{K}$-convergence in product space}


In this Section, we discuss $\mathcal{I}^\mathcal{K}$-convergence in product space and its analogous study for semi-open sets. Suppose that $\mathcal{I}$, $\mathcal{K}$ be two ideals on $\mathbb{N}$ and $\{x(n) \}_{n \in \mathbb{N}}$ be a function (generalization of a sequence) in a topological space $X$. Proceeding theorem ensures that a function is $\mathcal{I}^\mathcal{K}$-convergent in the product space if and only if each component function is $\mathcal{I}^\mathcal{K}$-convergent in each component space.
	
\begin{theorem} \label{PSIK}
	Let $\{(X_\alpha, \tau_\alpha)\}_{\alpha \in \Lambda} $ be a family of topological spaces. Let $X= \prod_{\alpha \in \Lambda}^{} X_\alpha$ be the product space and $\{x_\alpha(n) \}$ be a $X_\alpha$-valued sequence, for each $\alpha \in \Lambda$. Then $\{x_\alpha(n)\}$ is $\mathcal{I}^\mathcal{K}$-convergent to $p_\alpha$, for all $\alpha \in \Lambda$ if and only if $\{(x_\alpha(n))_{\alpha \in \Lambda} \}$ is $\mathcal{I}^\mathcal{K}$-convergent to $(p_\alpha)_{\alpha \in \Lambda}$.
\end{theorem}
	
	\begin{proof}
		Consider the product space $X= \prod X_\alpha$, of the topological spaces $\{X_\alpha \}_{\alpha \in \Lambda}$. 
		Suppose that for each $\alpha \in \Lambda$, $\{x_\alpha(n) \}$ is a sequence in $X_\alpha$, $\mathcal{I}^\mathcal{K}$-convergent to $p_\alpha \in X_\alpha$. We claim that $(x_\alpha (n))_{\alpha \in \Lambda}$ is $\mathcal{I}^\mathcal{K}$-convergent to $(p_\alpha)_{\alpha \in \Lambda} \in X$. 
		
		\vspace{2mm}
		Contrapositively, suppose that $(x_\alpha (n))_{\alpha \in \Lambda}$ is not $\mathcal{I}^\mathcal{K}$-convergent to $(p_\alpha)_{\alpha \in \Lambda}$. Then, \linebreak for each $M \in \mathcal{I}^*$, there exists an open set $U$ containing $(p_\alpha)_{\alpha \in \Lambda}$ such that \linebreak $\{n \in M: (x_\alpha(n))_{\alpha \in \Lambda} \notin U \} \notin \mathcal{K}$. Now, there exists a basic open set $B= \prod_{\alpha \in \Lambda} B_\alpha \subset U$ \linebreak containing $(p_\alpha)_{\alpha \in \Lambda}$, where $B_\alpha =X_\alpha$, except for finite no of $\alpha \in \Lambda$, say $\alpha_1, \alpha_2,..., \alpha_k$. So, we have 
		\begin{center}
			 $\{n \in M: (x_\alpha(n))_{\alpha \in \Lambda} \notin U \} \subset \{n \in M: (x_\alpha(n))_{\alpha \in \Lambda} \notin  \prod_{\alpha \in \Lambda} B_\alpha \} \}$.
		\end{center}
		Since, $\{n \in M: (x_\alpha(n))_{\alpha \in \Lambda} \notin U \} \notin \mathcal{K}$, that implies
		\begin{center}
			$\{n \in M: (x_\alpha(n))_{\alpha \in \Lambda} \notin  \prod_{\alpha \in \Lambda} B_\alpha \} \notin \mathcal{K}$.
		\end{center}
	Again, we observe that
	\begin{align*}
	\{n \in M: (x_\alpha(n))_{\alpha \in \Lambda} \notin  \prod_{\alpha \in \Lambda} B_\alpha \} 
	=& \bigcup_{ \alpha \in \Lambda}  \{n \in M: x_{\alpha}(n) \notin B_\alpha \} \\
	=& \bigcup_{i=1}^{k}  \{n \in M: x_{\alpha_i}(n) \notin B_{\alpha_i} \}.
		\end{align*}
		Therefore, we have $\bigcup_{i=1}^{k}  \{n \in M: x_{\alpha_i}(n) \notin B_{\alpha_i} \} \notin \mathcal{K}$. Now there must exists $\alpha_0 \in \{\alpha_1, \alpha_2, ..., \alpha_k \}$ such that $\{n \in M: x_{\alpha_0}(n) \notin \prod_{\alpha \in \Lambda} B_\alpha \} \notin \mathcal{K}$. Otherwise, if $\{n \in M: (x_{\alpha_i}(n))_{{\alpha_{i}} \in \Lambda} \notin  \prod_{\alpha \in \Lambda} B_\alpha \} \in \mathcal{K}$ for each $\alpha_i \in \{\alpha_1, \alpha_2, ..., \alpha_k \}$, then
		\begin{center}
			$\bigcup_{i=1}^{k}  \{n \in M: (x_{\alpha_i}(n))_{{\alpha_{i}} \in \Lambda} \notin  \prod_{\alpha \in \Lambda} B_\alpha \} \in \mathcal{K}$
		\end{center}
		 This is a contradiction. That means $x_{\alpha_0} (n)$ is not $\mathcal{I}^\mathcal{K}$-convergent to $p_{\alpha_0}$ and that contradicts our assumption. Thus, we must have $(x_\alpha (n))_{\alpha \in \Lambda}$ is $\mathcal{I}^\mathcal{K}$-convergent to $(p_\alpha)_{\alpha \in \Lambda} \in X$.
	
	\vspace{2mm}
	\noindent Conversely, suppose that $(x_\alpha (n))_{\alpha \in \Lambda}$ is $\mathcal{I}^\mathcal{K}$-convergent to $(p_\alpha)_{\alpha \in \Lambda}$. For each $\alpha \in \Lambda$, let $P_\alpha$ be  projection mapping from $X = \prod_{\alpha \in \Lambda} X_\alpha$ to $X_\alpha$. Again, $P_\alpha$ is a continuous mapping for each $\alpha \in \Lambda$. So, by Theorem \ref{IKc}, for each $\alpha \in \Lambda$, $P_\alpha$ preserves $\mathcal{I}^\mathcal{K}$-convergence. Thus, $P_\alpha((x_\alpha(n))_{\alpha \in \Lambda}) = x_\alpha(n)$, being a continuous image of $(x_\alpha(n)_{\alpha \in \Lambda}$, is $\mathcal{I}^\mathcal{K}$-convergent to $P_\alpha((p_\alpha)_{\alpha \in \Lambda }) = p_\alpha$.
	\end{proof}
	
	Recall that if $(\mathcal{I}_\alpha)_{\alpha \in \Lambda}$ is a chain of ideals on a set $S$, then $\bigcup_{ \alpha \in \Lambda} \mathcal{I}_\alpha$ is an ideal on $X$ \cite{SV15}. We may have generalize the above theorem for a chain of ideals as follows.
	
	\begin{theorem}
		Let $(\mathcal{K}_\alpha)_{\alpha \in \Lambda}$ be a chain of proper ideals on $\mathbb{N}$ and $\mathcal{K}= \bigcup_{ \alpha \in \Lambda} \mathcal{K}_\alpha$, be a ideal on $\mathbb{N}$. Let \linebreak $X= \prod_{\alpha \in \Lambda} X_\alpha$ be the product space for a indexed family of topological spaces $\{X_\alpha \}_{\alpha \in \Lambda}$ and $\{x_\alpha(n) \}$ be a $X_\alpha$-valued sequence, for each $\alpha \in \Lambda$. Then $\{x_\alpha(n)\}$ is $\mathcal{I}^{\mathcal{K}_\alpha}$-convergent to $p_\alpha$, for all $\alpha \in \Lambda$ if and only if $\{(x_\alpha(n))_{\alpha \in \Lambda} \}$ is $\mathcal{I}^\mathcal{K}$-convergent to $(p_\alpha)_{\alpha \in \Lambda}$.
	\end{theorem}
		
	\begin{proof}
		We omit the proof for being analogous to that of Theorem \ref{PSIK}. Using the fact that $(\mathcal{K}_\alpha)_{\alpha \in \Lambda}$ is a chain of proper ideals on a set $\mathbb{N}$ and $\mathcal{K}= \bigcup_{ \alpha \in \Lambda} \mathcal{K}_\alpha$ in Theorem \ref{PSIK}, we conclude the rest.
	\end{proof}
 
	\begin{corollary}
		Let $\{(X_\alpha, \tau_\alpha)\}_{\alpha \in \Lambda} $ be an family of topological spaces. Let \linebreak $X= \prod_{\alpha \in \Lambda}^{} X_\alpha$ be the product space and $\{x_\alpha(n) \}$ be a $X_\alpha$-valued sequence, for each $\alpha \in \Lambda$. Then $\{x_\alpha(n)\}$ is $\mathcal{S}$-$\mathcal{I}^\mathcal{K}$-convergent to $p_\alpha$, for all $\alpha \in \Lambda$ if and only if $\{(x_\alpha(n))_{\alpha \in \Lambda} \}$ is $\mathcal{S}$-$\mathcal{I}^\mathcal{K}$-convergent to $(p_\alpha)_{\alpha \in \Lambda}$.
	\end{corollary}
	
	\begin{corollary}
		Let $(\mathcal{K}_\alpha)_{\alpha \in \Lambda}$ be a chain of proper ideals on $\mathbb{N}$ and $\mathcal{K}= \bigcup_{ \alpha \in \Lambda} \mathcal{K}_\alpha$, be an ideal on $\mathcal{N}$. Let \linebreak $X= \prod_{\alpha \in \Lambda} X_\alpha$ be the product space for a family of topological spaces $\{X_\alpha \}_{\alpha \in \Lambda}$ and $\{x_\alpha(n) \}$ be a $X_\alpha$-valued sequence, for each $\alpha \in \Lambda$. Then $\{x_\alpha(n)\}$ is $\mathcal{S}$-$\mathcal{I}^{\mathcal{K}_\alpha}$-convergent to $p_\alpha$, for all $\alpha \in \Lambda$ if and only if $\{(x_\alpha(n))_{\alpha \in \Lambda} \}$ is $\mathcal{S}$-$\mathcal{I}^\mathcal{K}$-convergent to $(p_\alpha)_{\alpha \in \Lambda}$.
	\end{corollary}

\section*{ Acknowledgement}
The first author would like to thank the University Grants Comission (UGC) for awarding the junior research fellowship vide UGC-Ref. No.: 1115/(CSIR-UGC NET DEC. 2017), India.

\end{document}